\documentclass[12pt, reqno, twoside, letterpaper]{amsart}

\usepackage{amsmath,amssymb,amsbsy,amsfonts,amsthm,latexsym,%mathabx,
	amsopn,amstext,amsxtra,euscript,amscd,stmaryrd,mathrsfs,
	cite,array}

\numberwithin{equation}{section}

\usepackage{color}
\usepackage[usenames,dvipsnames,svgnames,table]{xcolor}

\def\eps{\varepsilon}
\def\mand{\qquad\mbox{and}\qquad}
\def\fl#1{\left\lfloor#1\right\rfloor}

\def\({\left(}
\def\){\right)}

\newcommand{\e}{\ensuremath{\mathbf{e}}}

\newcommand{\cA}{\ensuremath{\mathcal{A}}}
\newcommand{\cB}{\ensuremath{\mathcal{B}}}

\newcommand{\cS}{\ensuremath{\mathcal{S}}}

\newcommand{\cX}{\ensuremath{\mathcal{X}}}

\newcommand{\sN}{\ensuremath{\mathscr{N}}}

\newcommand{\sP}{\ensuremath{\mathscr{P}}}

\newcommand{\NN}{\ensuremath{\mathbb{N}}}
\newcommand{\PP}{\ensuremath{\mathbb{P}}}

\newcommand{\RR}{\ensuremath{\mathbb{R}}}
\newcommand{\ZZ}{\ensuremath{\mathbb{Z}}}

\usepackage{color}
\usepackage{hyperref}
\hypersetup{
	pdftoolbar=true,
	pdfmenubar=false,
	pdffitwindow=false,
	pdfstartview={FitH},
	pdftitle={Piatetski-Shapiro Primes in the intersection of multiple Beatty sequences}, % title
	pdfauthor={Victor Zhenyu Guo, Jinjiang Li, Min Zhang}, % authors
	pdfcreator={Victor Zhenyu Guo, Jinjiang Li, Min Zhang}, % authors
	pdfproducer={Victor Zhenyu Guo, Jinjiang Li, Min Zhang}, % authors
	pdfnewwindow=true,
	colorlinks=true,
	linkcolor={black},
	citecolor={black},
	filecolor={black},
	urlcolor={black},
}

\usepackage[text={144mm,233mm},centering]{geometry} % Text height : text width approximates Golden Ratio.
\usepackage{setspace}
\usepackage{amsthm}
\newtheoremstyle{customthm}
{1em}                    % ABOVESPACE
{1em}                    % BELOWSPACE
{\itshape}               % BODYFONT
{}                       % INDENT
{\scshape}               % HEADFONT
{.}                      % HEADPUNCT
{5pt plus 1pt minus 1pt} % HEADSPACE
{}                       % CUSTOM-HEAD-SPEC

\newtheoremstyle{customrem}
{1em}                    % ABOVESPACE
{1em}                    % BELOWSPACE
{}                       % BODYFONT
{}                       % INDENT
{\scshape}               % HEADFONT
{.}                      % HEADPUNCT
{5pt plus 1pt minus 1pt} % HEADSPACE
{}                       % CUSTOM-HEAD-SPEC

\theoremstyle{customthm}

\newtheorem{X}{X}[section]
\newtheorem{theorem}[X]{Theorem}

\newtheorem{lemma}[X]{Lemma}

\theoremstyle{customrem}

\usepackage{etoolbox}
\AtEndEnvironment{remark}{\null\hfill\qedsymbol}

\AtEndEnvironment{definition}{\null\hfill\qedsymbol}

\renewcommand{\le}{\ensuremath{\leqslant}}
\renewcommand{\ge}{\ensuremath{\geqslant}}
\def\fl#1{\left\lfloor#1\right\rfloor}

\renewcommand{\pod}[1]{\mathchoice
	{\allowbreak \if@display \mkern 5mu\else \mkern 5mu\fi (#1)}
	{\allowbreak \if@display \mkern 5mu\else \mkern 5mu\fi (#1)}
	{\mkern4mu(#1)}
	{\mkern4mu(#1)}
}

\newcommand*{\defeq}{\mathrel{\vcenter{\baselineskip0.5ex \lineskiplimit0pt
			\hbox{\scriptsize.}\hbox{\scriptsize.}}}%
	=}

\DeclareSymbolFont{EUEX}{U}{euex}{m}{n}

\DeclareSymbolFont{euexlargesymbols}{U}{euex}{m}{n}
\DeclareMathSymbol{\intop}{\mathop}{euexlargesymbols}{"52}
\def\int{\intop\nolimits}

\DeclareSymbolFont{euexsymbols}     {U}{euex}{m}{n}
\DeclareMathSymbol{\smallint}{\mathop}{euexsymbols}{"52}

\usepackage{ifthen}
\usepackage{xifthen}

\def\sums{
	\@ifnextchar[
	{\sums@i}
	{\ensuremath{\sum}}
}
\def\sums@i[#1]{
	\@ifnextchar[
	{\sums@ii{#1}}
	{\ensuremath{\sum_{#1}}}
}
\def\sums@ii#1[#2]{
	\@ifnextchar[
	{\sums@iii{#1}{#2}}
	{\ensuremath{\sum_{\substack{#1 \\ #2}}}}
}
\def\sums@iii#1#2[#3]{
	\@ifnextchar[
	{\sums@iv{#1}{#2}{#3}}
	{\ensuremath{\sum_{\substack{#1 \\ #2 \\ #3}}}}
}
\def\sums@iv#1#2#3[#4]{
	\@ifnextchar[
	{\sums@v{#1}{#2}{#3}{#4}}
	{\ensuremath{\sum_{\substack{#1 \\ #2 \\ #3 \\ #4}}}}
}
\def\sums@v#1#2#3#4[#5]{
	{\ensuremath{\sum_{\substack{#1 \\ #2 \\ #3 \\ #4 \\ #5}}}}
}

\def\sumss[#1]{
	\@ifnextchar[
	{\sumss@i[#1]}
	{
		\ifthenelse{\isempty{#1}}
		{\ensuremath{\sum}}
		{
			\ifthenelse{\equal{#1}{'}}
			{\ensuremath{\sideset{}{^{\prime}}{\sum}}}
			{\ensuremath{\sideset{}{^{#1}}{\sum}}}
		}
	}
}
\def\sumss@i[#1][#2]{
	\@ifnextchar[
	{\sumss@ii[#1]{#2}}
	{
		\ifthenelse{\isempty{#1}}
		{\ensuremath{\sum_{#2}}}
		{
			\ifthenelse{\equal{#1}{'}}
			{\ensuremath{\sideset{}{^{\prime}}{\sum}_{#2}}}
			{\ensuremath{\sideset{}{^{#1}}{\sum}_{#2}}}
		}
	}
}
\def\sumss@ii[#1]#2[#3]{
	\@ifnextchar[
	{\sumss@iii[#1]{#2}{#3}}
	{
		\ifthenelse{\isempty{#1}}
		{\ensuremath{\sum_{\substack{#2 \\ #3}}}}
		{
			\ifthenelse{\equal{#1}{'}}
			{\ensuremath{\sideset{}{^{\prime}}{\sum}_{\substack{#2 \\ #3}}}}
			{\ensuremath{\sideset{}{^{#1}}{\sum}_{\substack{#2 \\ #3}}}}
		}
	}
}
\def\sumss@iii[#1]#2#3[#4]{
	\@ifnextchar[
	{\sumss@iv[#1]{#2}{#3}{#4}}
	{
		\ifthenelse{\isempty{#1}}
		{\ensuremath{\sum_{\substack{#2 \\ #3 \\ #4}}}}
		{
			\ifthenelse{\equal{#1}{'}}
			{\ensuremath{\sideset{}{^{\prime}}{\sum}_{\substack{#2 \\ #3 \\ #4}}}}
			{\ensuremath{\sideset{}{^{#1}}{\sum}_{\substack{#2 \\ #3 \\ #4}}}}
		}
	}
}
\def\sumss@iv[#1]#2#3#4[#5]{
	\@ifnextchar[
	{\sumss@v[#1]{#2}{#3}{#4}{#5}}
	{
		\ifthenelse{\isempty{#1}}
		{\ensuremath{\sum_{\substack{#2 \\ #3 \\ #4 \\ #5}}}}
		{
			\ifthenelse{\equal{#1}{'}}
			{\ensuremath{\sideset{}{^{\prime}}{\sum}_{\substack{#2 \\ #3 \\ #4 \\ #5}}}}
			{\ensuremath{\sideset{}{^{#1}}{\sum}_{\substack{#2 \\ #3 \\ #4 \\ #5}}}}
		}
	}
}
\def\sumss@v[#1]#2#3#4#5[#6]{
	{\ifthenelse{\isempty{#1}}
		{\ensuremath{\sum_{\substack{#2 \\ #3 \\ #4 \\ #5 \\ #6 }}}}
		{
			\ifthenelse{\equal{#1}{'}}
			{\ensuremath{\sideset{}{^{\prime}}{\sum}_{\substack{#2 \\ #3 \\ #4 \\ #5 \\ #6 }}}}
			{\ensuremath{\sideset{}{^{#1}}{\sum}_{\substack{#2 \\ #3 \\ #4 \\ #5 \\ #6 }}}}
		}
	}
}

\def\sumsstxt[#1]{
	\@ifnextchar[
	{\sumsstxt@i[#1]}
	{
		\ifthenelse{\isempty{#1}}
		{\ensuremath{\textstyle\sum}}
		{
			\ifthenelse{\equal{#1}{'}}
			{\ensuremath{\sideset{}{^{\prime}}{\textstyle\sum}}}
			{\ensuremath{\sideset{}{^{#1}}{\textstyle\sum}}}
		}
	}
}
\def\sumsstxt@i[#1][#2]{
	\@ifnextchar[
	{\sumsstxt@ii[#1]{#2}}
	{
		\ifthenelse{\isempty{#1}}
		{\ensuremath{\textstyle\sum_{#2}}}
		{
			\ifthenelse{\equal{#1}{'}}
			{\ensuremath{\sideset{}{^{\prime}}{\textstyle\sum}_{#2}}}
			{\ensuremath{\sideset{}{^{#1}}{\textstyle\sum}_{#2}}}
		}
	}
}
\def\sumsstxt@ii[#1]#2[#3]{
	\@ifnextchar[
	{\sumsstxt@iii[#1]{#2}{#3}}
	{
		\ifthenelse{\isempty{#1}}
		{\ensuremath{\textstyle\sum_{\substack{#2 \\ #3}}}}
		{
			\ifthenelse{\equal{#1}{'}}
			{\ensuremath{\sideset{}{^{\prime}}{\textstyle\sum}_{\substack{#2 \\ #3}}}}
			{\ensuremath{\sideset{}{^{#1}}{\textstyle\sum}_{\substack{#2 \\ #3}}}}
		}
	}
}
\def\sumsstxt@iii[#1]#2#3[#4]{
	\@ifnextchar[
	{\sumsstxt@iv[#1]{#2}{#3}{#4}}
	{
		\ifthenelse{\isempty{#1}}
		{\ensuremath{\textstyle\sum_{\substack{#2 \\ #3 \\ #4}}}}
		{
			\ifthenelse{\equal{#1}{'}}
			{\ensuremath{\sideset{}{^{\prime}}{\textstyle\sum}_{\substack{#2 \\ #3 \\ #4}}}}
			{\ensuremath{\sideset{}{^{#1}}{\textstyle\sum}_{\substack{#2 \\ #3 \\ #4}}}}
		}
	}
}
\def\sumsstxt@iv[#1]#2#3#4[#5]{
	{\ifthenelse{\isempty{#1}}
		{\ensuremath{\textstyle\sum_{\substack{#2 \\ #3 \\ #4 \\ #5}}}}
		{
			\ifthenelse{\equal{#1}{'}}
			{\ensuremath{\sideset{}{^{\prime}}{\textstyle\sum}_{\substack{#2 \\ #3 \\ #4 \\ #5}}}}
			{\ensuremath{\sideset{}{^{#1}}{\textstyle\sum}_{\substack{#2 \\ #3 \\ #4 \\ #5}}}}
		}
	}
}

\def\prods{
	\@ifnextchar[
	{\prods@i}
	{\ensuremath{\prod}}
}
\def\prods@i[#1]{
	\@ifnextchar[
	{\prods@ii{#1}}
	{\ensuremath{\prod_{#1}}}
}
\def\prods@ii#1[#2]{
	\@ifnextchar[
	{\prods@iii{#1}{#2}}
	{\ensuremath{\prod_{\substack{#1 \\ #2}}}}
}
\def\prods@iii#1#2[#3]{
	\@ifnextchar[
	{\prods@iv{#1}{#2}{#3}}
	{\ensuremath{\prod_{\substack{#1 \\ #2 \\ #3}}}}
}
\def\prods@iv#1#2#3[#4]{
	{\ensuremath{\prod_{\substack{#1 \\ #2 \\ #3 \\ #4}}}}
}
\allowdisplaybreaks

\title[Beatty primes from fractional powers of almost--primes]
      {Beatty primes from fractional powers of almost--primes}

\author[Victor Zhenyu Guo]{Victor Zhenyu Guo}

\address{School of Mathematics and Statistics, Xi'an Jiaotong University, Xi'an, Shaanxi, 710049, P. R. China}

\email{guozyv@xjtu.edu.cn}

\author[Jinjiang Li]{Jinjiang Li}

\address{Department of Mathematics, China University of Mining and Technology, Beijing, 100083, P. R. China}

\email{jinjiang.li.math@gmail.com}

\author[Min Zhang]{Min Zhang}

\address{School of Applied Science, Beijing Information Science and Technology University, Beijing, 100192, P. R. China}

\email{min.zhang.math@gmail.com}

\begin{document}
\footnotetext[1]{Jinjiang Li is the corresponding author.}

\begin{abstract}
Let $\alpha>1$ be irrational and of finite type, $\beta\in\mathbb{R}$. In this paper, it is proved that for $R\geqslant13$ and any fixed $c\in(1,c_R)$, there exist infinitely many primes in the intersection of Beatty sequence $\mathcal{B}_{\alpha,\beta}$ and $\lfloor n^c\rfloor$, where $c_R$ is an explicit constant depending on $R$ herein, $n$ is a natural number with at most $R$ prime factors, counted with multiplicity.
\end{abstract}

\maketitle

{\textbf{Keywords}}: Beatty sequence; Piatetski--Shapiro prime; exponential sum; almost--prime

{\textbf{MR(2020) Subject Classification}}: 11N05, 11L07, 11N80, 11B83, 11N36

%\newcommand{\tind}[1]{\ensuremath{\widetilde{\mathbf{1}}_{#1}}}

%%%%%%%%%%%%%%%%%%%%%%%%
%%%%% PAPER BEGINS %%%%%
%%%%%%%%%%%%%%%%%%%%%%%%

\section{Introduction}
\label{sec:intro}

For fixed real numbers $\alpha$ and $\beta$, the associated non--homogeneous Beatty sequence is the sequence of integers defined by
$$
\cB_{\alpha,\beta} \defeq \(\fl{\alpha n+\beta}\)_{n=1}^\infty,
$$
where $\fl{t}$ denotes the integral part of any $t\in\RR$. Such sequences are also called generalized arithmetic progressions. If $\alpha$ is irrational, it follows from a classical exponential sum estimate of Vinogradov~\cite{Vino} that $\cB_{\alpha,\beta}$ contains infinitely many prime numbers. In fact, one has the asymptotic formula
$$
\#\big\{\text{prime~}p\le x:p\in\cB_{\alpha,\beta}\big\}=\alpha^{-1}\pi(x)(1+o(1))
$$
as $x\to\infty$, where $\pi(x)$ is the prime counting function.

For $1<c\not\in\mathbb{N}$, the Piatetski--Shapiro sequences are sequences of the form
$$
\mathscr{N}^{(c)} \defeq (\fl{n^c})_{n=1}^\infty.
$$
Such sequences have been named in honor of Piatetski--Shapiro, who, in \cite{PS}, proved that $\mathscr{N}^{(c)}$ contains infinitely many primes for $c\in(1,\frac{12}{11})$. To be specific, Piatetski--Shapiro showed that, for $c\in(1,\frac{12}{11})$, the counting function
$$
\pi^{(c)}(x) \defeq \# \big\{\text{\rm prime~}p\le x : p\in\mathscr{N}^{(c)}\big\}
$$
satisfies the asymptotic formula
$$
\pi^{(c)}(x) = \frac{x^{1/c}}{\log x}(1+o(1))
$$
as $x\to\infty$. The range of $c$ in the above asymptotic formula, in which it is known that $\sN^{(c)}$ contains infinitely many primes, has been extended many times over the years and is currently known to hold for all $c\in(1,\frac{2817}{2426})$ thanks to Rivat and Sargos \cite{RiSa}. Rivat and Wu \cite{RiWu} also showed that there exist infinitely many Piatetski--Shapiro primes for $c\in(1, \frac{243}{205})$ by showing a lower bound of $\pi^{(c)}(x)$ with the excepted order of
magnitude. It is conjectured that the above asymptotic formula should be held for all larger values of $c\not\in\mathbb{N}$.  We remark that if $c\in(0,1)$, then $\sN^{(c)}$ contains all natural numbers, and hence all prime numbers in particular. The estimation of Piatetski--Shapiro prime is an approximation of the well--known conjecture that there exist infinitely many primes of the form $n^2+1$.

For any natural number $R \ge 1$, let $\sP_R$ denote the set of $R$--almost primes, i.e., the set of natural numbers with at most $R$ prime factors, counted according to multiplicity. By the motivation of finding a subset of Piatetski--Shapiro primes related to almost--primes, Baker, Banks, Guo and Yeager \cite{BBGY} proved that, for any fixed $c\in(1,\frac{77}{76})$, there holds
$$
\#\big\{n\le x:n\in\sP_8\text{~and~}\fl{n^c}\textrm{~is prime}\big\}
\gg\frac{x}{\log x},
$$
where the implied constant in the symbol $\gg$ depends only on $c$. In this paper, we consider the primes generated from the
intersection of a Beatty sequence and a set with above type, and establish the following theorem.

\begin{theorem}
\label{thm:main}
Suppose that $\alpha,\beta\in\mathbb{R}$. Let $\alpha>1$ be irrational and of finite type.
For $R\geqslant13$ and any fixed $c\in(1,c_R)$ with $c_R \defeq \frac{96R-12}{88R+85}$, there holds
\begin{equation*}
\#\big\{p\leqslant x: n\in \mathscr{P}_R, \fl{n^c} = p \in \cB_{\alpha, \beta}, \textrm{ p is prime}  \big\}
\gg\frac{x^{1/c}}{\log x},
\end{equation*}
where the implied constant in the symbol $\gg$ depends only on $\alpha,\beta$ and $c$.	
\end{theorem}
We give a short table for the admissible pairs $(R, c_R)$.
\begin{center}\label{tab:c_R}
	\begin{tabular}{|c|c||c|c||c|c|}
		\hline
		\vphantom{\Big|}$R$  &  $c_R$ & $R$  &  $c_R$ & $R$  &  $c_R$ \\
		\hline
		$13$ & $1.0056$ & $16$ & $1.0207$ & $19$ & $1.0313$ \\
		$14$ & $1.0113$ & $17$ & $1.0246$ & $20$ & $1.0341$ \\
		$15$ & $1.0163$ & $18$ & $1.0281$ & $21$ & $1.0367$ \\
		\hline
	\end{tabular}
\end{center}

\section{Preliminaries}

\subsection{Notation}

We denote by $\fl{t}$ and $\{t\}$ the integral part and the fractional part of $t$, respectively.
As usual, we put
$$
\e(t)\defeq e^{2\pi it}.
$$
Throughout this paper, we make considerable use of the sawtooth function, which is defined by
$$
\psi(t) \defeq t-\fl{t}-\frac{1}{2} = \{t\}-\frac{1}{2}.
$$
The notation $\| t \|$ is used to denote the distance
from the real number $t$ to the nearest integer, i.e.,
$$
\| t\| \defeq\min_{n\in\ZZ}|t-n|.
$$

Let $\PP$ denote the set of primes in $\NN$. The letter $p$ always denotes a prime. For the Beatty sequence $(\fl{\alpha n + \beta})_{n=1}^\infty$, we denote $a \defeq \alpha^{-1}$. We represent $\gamma \defeq c^{-1}$ for the Piatetski--Shapiro sequence $(\fl{n^c})_{n=1}^\infty$. We use notation of the form $m\sim M$ as an abbreviation for $M< m\le 2M$.

Throughout this paper, implied constants in symbols $O$, $\ll$
and $\gg$ may depend (where obvious) on the parameters $\alpha,\beta,c,\eps$, but are absolute otherwise. For given
functions $F$ and $G$, the notations $F\ll G$, $G\gg F$ and $F=O(G)$ are all equivalent to the statement that the inequality
$|F|\le \mathcal{C}|G|$ holds with some constant $\mathcal{C}>0$.

\subsection{The weighted sieve}

As is mentioned below, the following notation plays a crucial role in our arguments. We specify it to the form
which is suitable for our applications. Plainly, it is based on a result of Greaves \cite{Greaves} which relates the level of distribution to $R$--almost primality. More precisely, we say that an $X$--element set of integers $\cA$ has a level of distribution $D$, if there holds
$$
\sum_{d \le D} \left|\big|\{a\in \cA, a \equiv 0  \pmod d\}\big|
- \frac{\omega(d)}{d}X\right| \ll  \frac{X}{\log^2X}
$$
for some given multiplicative function $\omega(d)$.
Moreover, we have the following result, which is \cite[Chapter~5, Proposition~1]{Greaves}.

\begin{lemma}\label{lem:Greaves}
Let $R$ be an integer with $R\geqslant2$. Suppose that $\cA$ is an $X$--element set of positive integers with a level of distribution $D$ and degree $g$ in the sense that
$$
	a\in \cA \Longrightarrow a<D^g
	$$
with some real number $g<R-\delta_R$. Then
	$$
	\#\big\{a \in \cA:\text{\rm $a$~is an $R$--almost prime}\big\}
	\gg_g  \frac{X}{\log X}\,.
	$$
\end{lemma}

\subsection{Type of an irrational number}

For any irrational number $\alpha$, we define its type $\tau=\tau(\alpha)$ by the following definition
$$
\tau\defeq\sup\Big\{t\in\RR:\liminf\limits_{n\to\infty} n^t \|\alpha n\|=0\Big\}.
$$
Using Dirichlet's approximation theorem, one sees
that $\tau\ge 1$ for every irrational number $\alpha$. Thanks to the work of Khinchin \cite{Khin} and Roth \cite{Roth1,Roth2}, it is known that $\tau=1$ for almost all real numbers, in the sense of the Lebesgue measure, and for all irrational
algebraic numbers, respectively. Moreover, if $\alpha$ is an irrational number of type $\tau<\infty$,
then so are $\alpha^{-1}$ and $n\alpha^{-1}$ for all integer $n\ge 1$.

\subsection{Technical lemmas}

We need the following well--known approximation of Vaaler~\cite{Vaal}.

\begin{lemma}
\label{lem:Vaaler}
For any $H\ge 1$, there exist numbers $a_h,b_h$ such that
$$
\bigg|\psi(t)-\sum_{0<|h|\le H}a_h\,\e(th)\bigg|
\le\sum_{|h|\le H}b_h\,\e(th),\qquad
a_h\ll\frac{1}{|h|}\,,\qquad b_h\ll\frac{1}{H}\,.
$$
\end{lemma}

\begin{lemma}
\label{lem:index}
For a bounded function $f$ and $N'\sim N$, we have
$$
\sum_{N < p\le N'} f(p) \ll \frac{1}{\log N} \max_{N<N_1\le2N}
\bigg|\sum_{N < n \le N_1} \Lambda(n)f(n)\bigg| + N^{1/2}.
$$
\end{lemma}
\begin{proof}
See the argument on page 48 of \cite{GraKol}.
\end{proof}

\begin{lemma}\label{lem:amn}
Suppose that
\begin{equation*}
\alpha=\frac{a}{q}+\frac{\theta}{q^2},
\end{equation*}
with $(a,q)=1, q\geqslant1, |\theta|\leqslant 1$. Then there holds
\begin{equation*}
\sum_{m\leqslant N}\Lambda(m)\mathbf{e}(m\alpha)\ll\big(Nq^{-1/2}+N^{4/5}+N^{1/2}q^{1/2}\big)(\log N)^4.
\end{equation*}
\end{lemma}
\begin{proof}
See Chapter 25 of Davenport \cite{Daven}.
\end{proof}

\begin{lemma}\label{lem:ttr}
Suppose that $a$ is a fixed irrational number of finite type $\tau < \infty$. Let $h\geqslant1$ and $m$ be integers.
Then we have
\begin{equation*}
\sum_{m\leqslant M}\Lambda(m)\mathbf{e}(ahm) \ll h^{1/2} M^{1-1/(2\tau) + \varepsilon} + M^{1-\eps}.
\end{equation*}
\end{lemma}

\begin{proof}
For any sufficiently small $\varepsilon>0$, we set $\varrho=\tau+\varepsilon$. Since $a$ is of type $\tau$, there exists some constant $\mathfrak{c}>0$ such that
\begin{equation}\label{eq:dis}
\|an\|>\mathfrak{c}n^{-\varrho},\qquad n\geqslant 1.
\end{equation}
For given $h$ with $0<h\leqslant H$, let $b/d$ be the convergent in the continued fraction expansion of $ah$, which has the largest denominator $d$ not exceeding $M^{1-\eta}$ for a sufficiently small positive number $\eta$. Then we derive that
\begin{equation}\label{eq:apr}
\bigg|ah-\frac{b}{d}\bigg|\leqslant\frac{1}{dM^{1-\eta}}\leqslant\frac{1}{d^2},
\end{equation}
which combined with \eqref{eq:dis} yields
$$
M^{-1+\eta}\geqslant|ahd-b|\geqslant\|ahd\|>\mathfrak{c}(hd)^{-\varrho}.
$$
Taking $C_0:=\mathfrak{c}^{1/\varrho}$, we obtain
\begin{equation}\label{eq:d}
d > C_0h^{-1}M^{1/\varrho - \eta/\varrho}.
\end{equation}
Combining~\eqref{eq:apr} and~\eqref{eq:d}, applying Lemma~\ref{lem:amn} and the fact that $d \le M^{1-\eta}$, we deduce that
\begin{align*}
\sum_{m \leqslant M}\Lambda(m)\mathbf{e}(ahm)
&\ll \big( M d^{-1/2} + M^{4/5} + M^{1/2} d^{1/2} \big) (\log M)^4 \\
&\ll \big( h^{1/2} M^{1-1/(2\varrho)+\eta/(2\varrho)}+ M^{4/5} +  M^{1-\eta/2} \big) (\log M)^4 \\
&\ll h^{1/2} M^{1-1/(2\tau) + \varepsilon} + M^{1-\eps}.
\end{align*}
This completes the proof of Lemma \ref{lem:ttr}.
\end{proof}

The following lemma gives a characterization of the numbers in the Beatty sequence $\cB_{\alpha, \beta}$.
\begin{lemma}\label{lem:Beatty}
A natural number $m$ has the form $\fl{\alpha n + \beta}$ if and only if $\chi_{\alpha, \beta} (m) = 1$, where $\chi_{\alpha, \beta} (m) \defeq \fl{-\alpha^{-1} (m-\beta)} - \fl{-\alpha^{-1} (m+1-\beta)}$.	
\end{lemma}
\begin{proof}
Note that an integer $m$ has the form $m= \fl{\alpha n+\beta} $ for some integer $n$ if and only if
\begin{equation*}
\frac{m-\beta}{\alpha}\leqslant n<\frac{m-\beta+1}{\alpha}.
\end{equation*}	
This completes the proof of Lemma \ref{lem:Beatty}.
\end{proof}

Finally, we use the following lemma, which provides a characterization of the numbers that
occur in the Piatetski--Shapiro sequence $\sN^{(c)}$.

\begin{lemma}
\label{lem:PS}
A natural number $m$ has the form $\fl{n^c}$ if and only if $\chi^{(c)}(m) = 1$, where
$\chi^{(c)}(m) \defeq \fl{-m^\gamma} - \fl{-(m+1)^\gamma}$.  Moreover,
$$
\chi^{(c)}(m)=\gamma m^{\gamma-1} + \psi(-(m+1)^\gamma) - \psi(-m^\gamma) + O(m^{\gamma-2}).
$$
\end{lemma}
\begin{proof}
The proof of Lemma \ref{lem:PS} is similar to that of Lemma \ref{lem:Beatty}, so we omit the details herein.
\end{proof}

\begin{lemma}
For $1<c<\frac{2817}{2426}$, there holds
\begin{equation} \label{eq:PSthm}
\pi^{(c)}(x) = \sum_{p \le x} \chi^{(c)}(p) = \frac{x^\gamma}{\log x}
+ O\bigg(\frac{x^\gamma}{\log^2 x}\bigg).
\end{equation}
\end{lemma}
\begin{proof}
See Theorem 1 of Rivat and Sargos \cite{RiSa}.	
\end{proof}

\subsection{Upper bound estimation of exponential sums}
We begin with the decomposition of the von Mangoldt function by Heath--Brown.

\begin{lemma}\label{Heath-Brown-identity}
	Let $z\geqslant1$ and $k\geqslant1$. Then, for any $n\leqslant2z^k$, there holds
	\begin{equation*}
		\Lambda(n)=\sum_{j=1}^k(-1)^{j-1}\binom{k}{j}\mathop{\sum\cdots\sum}_{\substack{n_1n_2\cdots n_{2j}=n\\
				n_{j+1},\dots,n_{2j}\leqslant z }}(\log n_1)\mu(n_{j+1})\cdots\mu(n_{2j}).
	\end{equation*}
\end{lemma}
\begin{proof}
	See the arguments on pp. 1366--1367 of Heath--Brown \cite{HB}.
\end{proof}

\begin{lemma}\label{compare}
	Suppose that
	\begin{equation*}
		L(H)=\sum_{i=1}^mA_iH^{a_i}+\sum_{j=1}^nB_jH^{-b_j},
	\end{equation*}
	where~$A_i,\,B_j,\,a_i$~and~$b_j$~are positive. Assume further that $H_1\leqslant H_2$. Then there exists
	some $\mathscr{H}$ with
	$H_1\leqslant\mathscr{H}\leqslant H_2$ and
	\begin{equation*}
		L(\mathscr{H})\ll \sum_{i=1}^mA_iH_1^{a_i}+\sum_{j=1}^nB_jH_2^{-b_j}+\sum_{i=1}^m\sum_{j=1}^n\big(A_i^{b_j}B_j^{a_i}\big)^{1/(a_i+b_j)}.
	\end{equation*}
	The implied constant depends only on $m$ and $n$.
\end{lemma}
\begin{proof}
	See Lemma 3 of Srinivasan \cite{Srin}.
\end{proof}

For a real number $m_1$, the sum of the form
\begin{equation*}
	\mathop{\sum_{k\sim K}\sum_{\ell\sim L}}_{KL\asymp x}a_kb_\ell\mathbf{e}\big(jd^{-1}(k\ell)^\gamma+m_1k\ell\big)
\end{equation*}
with $|a_k|\ll x^\varepsilon,|b_\ell|\ll x^\varepsilon$ for every fixed $\varepsilon>0$, it is usually called a ``Type I'' sum, denoted by $S_I(K,L)$, if $b_\ell=1$ or $b_\ell=\log\ell$; otherwise it is called a ``Type II'' sum, denoted by $S_{II}(K,L)$.

\begin{lemma}\label{derivative-estimate}
	Suppose that $f(x):[a, b]\to \mathbb{R}$ has continuous derivatives of arbitrary
	order on $[a,b]$, where $1\leqslant a<b\leqslant2a$. Suppose further that
	\begin{equation*}
		\big|f^{(j)}(x)\big|\asymp \lambda_j,\qquad j\geqslant1, \qquad x\in[a, b].
	\end{equation*}
	Then we have
	\begin{equation}\label{2nd-deri-estimate}
		\sum_{a<n\leqslant b}e\big(f(n)\big)\ll a\lambda_2^{1/2}+\lambda_2^{-1/2},
	\end{equation}
	and
	\begin{equation}\label{3rd-deri-estimate}
		\sum_{a<n\leqslant b}e\big(f(n)\big)\ll a\lambda_3^{1/6}+\lambda_3^{-1/3}.
	\end{equation}
\end{lemma}
\begin{proof}
	For (\ref{2nd-deri-estimate}), one can see Corollary 8.13 of Iwaniec and Kowalski \cite{IwKo}, or Theorem 5
	of Chapter 1 in Karatsuba \cite{Kara}. For (\ref{3rd-deri-estimate}), one can see Corollary 4.2 of
	Sargos \cite{Sarg}.
\end{proof}

\begin{lemma}\label{Type-I-sum}
Suppose that $|a_k|\ll 1,b_\ell=1$ or $\log\ell,KL\asymp x$. Then if $K\ll x^{1/2}$, there holds
\begin{equation*}
S_{I}(K,L)\ll|j|^{1/6}|d|^{-1/6}x^{\gamma/6+3/4+\eps}+|j|^{-1/3}|d|^{1/3}x^{1-\gamma/3+\eps}.
\end{equation*}
\end{lemma}
\begin{proof}
Set $f(\ell)=jd^{-1}(k\ell)^\gamma+m_1 k\ell$. It is easy to see that
\begin{equation*}
f'''(\ell)=\gamma(\gamma-1)(\gamma-2)jd^{-1}k^\gamma\ell^{\gamma-3}\asymp |j||d|^{-1} K^{\gamma}L^{\gamma-3}.
\end{equation*}
If $K\ll x^{1/2}$, then by (\ref{3rd-deri-estimate}) of Lemma \ref{derivative-estimate}, we deduce that
\begin{align*}
x^{-\eps} \cdot S_{I}(K,L)
\ll & \,\, \sum_{k\sim K} \Bigg|\sum_{\ell\sim L}\mathbf{e}\big(f(\ell)\big)\Bigg|
              \nonumber \\
\ll & \,\, \sum_{k\sim K} \Big( L \big( |j||d|^{-1}K^{\gamma}L^{\gamma-3} \big)^{1/6}
           + \big( |j||d|^{-1}K^{\gamma}L^{\gamma-3} \big)^{-1/3} \Big)
              \nonumber \\
\ll & \,\, |j|^{1/6}|d|^{-1/6}x^{\gamma/6+1/2}K^{1/2} + |j|^{-1/3}|d|^{1/3}x^{1-\gamma/3}
               \nonumber \\
\ll & \,\, |j|^{1/6}|d|^{-1/6}x^{\gamma/6+3/4}+|j|^{-1/3}|d|^{1/3}x^{1-\gamma/3},
\end{align*}
which completes the proof of Lemma \ref{Type-I-sum}.
\end{proof}

\begin{lemma}\label{Type-II-sum}
Suppose that $|a_k|\ll 1,|b_\ell|\ll1$ with $k\sim K,\ell\sim L$ and $KL\asymp x$. Then if $x^{1/2}\ll K\ll x^{39/50}$, there holds
\begin{align*}
S_{II}(K,L) \ll & \,\,  |j|^{1/4}|d|^{-1/4}x^{\gamma/4+5/8} + |j|^{-1/4}|d|^{1/4}x^{1-\gamma/4}
                        \nonumber \\
                & \,\, + x^{89/100} +|j|^{1/6}|d|^{-1/6}x^{\gamma/6+3/4}.
\end{align*}
\end{lemma}
\begin{proof}
Let $Q$, which satisfies $1 < Q < L$, be a parameter which will be chosen later. By the Weyl--van der Corput inequality
(see Lemma 2.5 of Graham and Kolesnik \cite{GraKol}), we have
\begin{align}\label{Weyl-inequality}
&\,\, \Bigg| \mathop{\sum_{k\sim K} \sum_{\ell\sim L}}_{KL\asymp x} a_k b_\ell
      \mathbf{e} \big(jd^{-1}(k\ell)^\gamma + m_1 k\ell \big)\Bigg|^2 \nonumber \\
\ll &\,\, K^2 L^2 Q^{-1} + KLQ^{-1} \sum_{\ell\sim L} \sum_{0<|q|\leqslant Q} \big|\mathfrak{S}(q;\ell)\big|,
\end{align}
where
\begin{equation*}
\mathfrak{S}(q;\ell)=\sum_{k\in\mathcal{I}(q;\ell)}\mathbf{e}\big(g(k)\big)
\end{equation*}
with
\begin{equation*}
g(k)=jd^{-1}k^\gamma\big(\ell^\gamma-(\ell+q)^\gamma\big)- m_1 kq.
\end{equation*}
It is easy to see that
\begin{equation*}
g''(k)=\gamma(\gamma-1)jd^{-1}k^{\gamma-2} \big(\ell^\gamma-(\ell+q)^\gamma\big)
 \asymp |j||d|^{-1} K^{\gamma-2} L^{\gamma-1} |q|.
\end{equation*}
By (\ref{2nd-deri-estimate}) of Lemma \ref{derivative-estimate}, we have
\begin{equation}\label{Type-II-inner}
\mathfrak{S}(q;\ell)\ll K\big(|j||d|^{-1} K^{\gamma-2} L^{\gamma-1} |q| \big)^{1/2}
+ \big(|j||d|^{-1} K^{\gamma-2}L^{\gamma-1}|q|\big)^{-1/2}.
\end{equation}
Putting (\ref{Type-II-inner}) into (\ref{Weyl-inequality}), we derive that
\begin{align*}
& \,\, \Bigg|\mathop{\sum_{k\sim K} \sum_{\ell\sim L}}_{KL\asymp x} a_k b_\ell
       \mathbf{e} \big( h (k\ell)^\gamma + m_1 k \ell + m_2 \big)\Bigg|^2
          \nonumber \\
\ll & \,\, K^2L^2Q^{-1}+KLQ^{-1}\sum_{\ell\sim L}\sum_{0<|q|\leqslant Q}
\big(|j|^{1/2}|d|^{-1/2}K^{\gamma/2}L^{\gamma/2-1/2}|q|^{1/2}
          \nonumber \\
    & \,\,+|j|^{-1/2}|d|^{1/2}K^{1-\gamma/2}L^{1/2-\gamma/2}|q|^{-1/2}\big)
          \nonumber \\
\ll & \,\,K^2L^2Q^{-1}+KLQ^{-1}\big(|j|^{1/2}|d|^{-1/2}K^{\gamma/2}L^{\gamma/2+1/2}Q^{3/2} \\
&\,\, + |j|^{-1/2}|d|^{1/2}K^{1-\gamma/2}L^{3/2-\gamma/2}Q^{1/2}\big)
          \nonumber \\
\ll & \,\, K^2L^2Q^{-1}+|j|^{1/2}|d|^{-1/2}K^{1+\gamma/2}L^{\gamma/2+3/2}Q^{1/2}
          \nonumber \\
    & \,\, +|j|^{-1/2}|d|^{1/2}K^{2-\gamma/2}L^{5/2-\gamma/2}Q^{-1/2}.
\end{align*}
By noting that $1\leqslant Q\leqslant L$, it follows from Lemma \ref{compare} that there exists an optimal $Q$ such that
\begin{align*}
    & \,\, \Bigg|\mathop{\sum_{k\sim K}\sum_{\ell\sim L}}_{KL\asymp x}a_kb_\ell
           \mathbf{e}\big(jd^{-1}(k\ell)^\gamma+m_1k\ell\big)\Bigg|^2
               \nonumber \\
\ll & \,\, |j|^{1/2}|d|^{-1/2}x^{\gamma/2+3/2}K^{-1/2}+Kx
            +|j|^{-1/2}|d|^{1/2}x^{2-\gamma/2}
               \nonumber \\
    & \,\,  +|j|^{1/3}|d|^{-1/3}x^{\gamma/3+5/3}K^{-1/3}+K^{-1/2}x^2,
\end{align*}
which implies
\begin{align*}
&\,\, \Bigg|\mathop{\sum_{k\sim K}\sum_{\ell\sim L}}_{KL\asymp x}a_kb_\ell
      \mathbf{e}\big(jd^{-1}(k\ell)^\gamma+m_1k\ell\big)\Bigg| \\
\ll & \,\, |j|^{1/4}|d|^{-1/4}x^{\gamma/4+3/4}K^{-1/4}+|j|^{-1/4}|d|^{1/4}x^{1-\gamma/4} + K^{1/2}x^{1/2} \\
    & \,\, +|j|^{1/6}|d|^{-1/6}x^{\gamma/6+5/6}K^{-1/6} + K^{-1/4}x \\
\ll & \,\, |j|^{1/4}|d|^{-1/4}x^{\gamma/4+5/8}+|j|^{-1/4}|d|^{1/4}x^{1-\gamma/4}+x^{89/100}
           +|j|^{1/6}|d|^{-1/6}x^{\gamma/6+3/4},
\end{align*}
by the condition $x^{1/2}\ll K\ll x^{39/50}$, which completes the proof of Lemma \ref{Type-II-sum}.
\end{proof}

\begin{lemma}\label{lem:exp}
Suppose that $x/2<t\leqslant x$. Then for any real number $m_1$, we have
\begin{align*}
    &\,\, \max_{x/2<t\leqslant x}\Bigg|\sum_{x/2<n\leqslant t}\Lambda(n)\mathbf{e}\big(jd^{-1}n^\gamma+m_1n\big) \Bigg|
                \nonumber \\
\ll &\,\, x^{\varepsilon} \Big( |j|^{1/6}|d|^{-1/6}x^{\gamma/6+3/4}+|j|^{-1/3}|d|^{1/3}x^{1-\gamma/3}
          +|j|^{1/4}|d|^{-1/4}x^{\gamma/4+5/8}
                 \nonumber \\
    &\,\, \qquad + |j|^{-1/4}|d|^{1/4}x^{1-\gamma/4}+x^{89/100}\Big).
\end{align*}	
\end{lemma}
\begin{proof}
By Heath--Brown's identity, i.e. Lemma \ref{Heath-Brown-identity}, with $k=3$, one can see that the exponential sum
\begin{equation*}
\max_{x/2<t\leqslant x}\Bigg|\sum_{x/2<n\leqslant t}\Lambda(n)\mathbf{e}\big(jd^{-1}n^\gamma+m_1n\big)\Bigg|
\end{equation*}
can be written as linear combination of $O(\log^6x)$ sums, each of which is of the form
\begin{align}\label{single-sum}
       \mathcal{T}^* :=
& \,\, \sum_{n_1\sim N_1}\cdots\sum_{n_6\sim N_6}(\log n_1)\mu(n_4)\mu(n_5)\mu(n_6) \nonumber \\
& \,\, \qquad\qquad\qquad \times \mathbf{e}\big(jd^{-1}(n_1n_2\cdots n_6)^\gamma+m_1(n_1n_2\cdots n_6)\big),
\end{align}
where $N_1N_2\cdots N_6\asymp x$; $2N_i\leqslant(2x)^{1/3},i=4,5,6$ and some $n_i$ may only take value $1$.
Therefore, it is sufficient for us to give upper bound estimate for each $\mathcal{T}^*$ defined as in (\ref{single-sum}). Next, we will consider three cases.

\noindent
\textbf{Case 1.} If there exists an $N_j$ such that $N_j\geqslant x^{1/2}$, then we must have $j\leqslant3$ for the fact that
$N_i\ll x^{1/3}$ with $i=4,5,6$. Let
\begin{equation*}
	k=\prod_{\substack{1\leqslant i\leqslant6\\ i\not=j}}n_i,\qquad \ell=n_j,\qquad
	K=\prod_{\substack{1\leqslant i\leqslant6\\ i\not=j}}N_i,\qquad L=N_j.
\end{equation*}
In this case, we can see that $\mathcal{T}^*$ is a sum of ``Type I'' satisfying $K\ll x^{1/2}$. By Lemma \ref{Type-I-sum},
we have
\begin{equation*}
x^{-\varepsilon}\cdot \mathcal{T}^*\ll |j|^{1/6}|d|^{-1/6}x^{\gamma/6+3/4}+|j|^{-1/3}|d|^{1/3}x^{1-\gamma/3}.
\end{equation*}

\noindent
\textbf{Case 2.} If there exists an $N_j$ such that $x^{11/50}\leqslant N_j<x^{1/2}$, then we take
\begin{equation*}
	k=\prod_{\substack{1\leqslant i\leqslant6\\ i\not=j}}n_i,\qquad \ell=n_j,\qquad
	K=\prod_{\substack{1\leqslant i\leqslant6\\ i\not=j}}N_i,\qquad L=N_j.
\end{equation*}
Thus, $\mathcal{T}^*$ is a sum of ``Type II'' satisfying $x^{1/2}\ll K\ll x^{39/50}$. By Lemma \ref{Type-II-sum},
we have
\begin{align*}
	         x^{-\varepsilon}\cdot  \mathcal{T}^*
 \ll & \,\, |j|^{1/4}|d|^{-1/4}x^{\gamma/4+5/8} + |j|^{-1/4}|d|^{1/4}x^{1-\gamma/4}
                        \nonumber \\
     & \,\, + x^{89/100} +|j|^{1/6}|d|^{-1/6}x^{\gamma/6+3/4}.
\end{align*}

\noindent
\textbf{Case 3.} If $N_j<x^{11/50}\,(j=1,2,3,4,5,6)$, without loss of generality, we assume that
$N_1\geqslant N_2\geqslant\cdots\geqslant N_6$. Let $r$ denote the natural number $j$ such that
\begin{equation*}
	N_1N_2\cdots N_{j-1}<x^{11/50},\qquad N_1N_2\cdots N_j\geqslant x^{11/50}.
\end{equation*}
Since $N_1<x^{11/50}$ and $N_6<x^{11/50}$, then $2\leqslant r\leqslant5$. Thus, we have
\begin{equation*}
	x^{11/50}\leqslant N_1N_2\cdots N_r=(N_1\cdots N_{r-1})\cdot N_r<x^{11/50}\cdot x^{11/50}<x^{1/2}.
\end{equation*}
Let
\begin{equation*}
	k=\prod_{i=r+1}^6n_i,\qquad \ell=\prod_{i=1}^rn_i,\qquad K=\prod_{i=r+1}^6N_i,\qquad L=\prod_{i=1}^rN_i.
\end{equation*}
At this time, $\mathcal{T}^*$ is a sum of ``Type II'' satisfying $x^{1/2}\ll K\ll x^{39/50}$. By Lemma \ref{Type-II-sum},
we have
\begin{align*}
	         x^{-\varepsilon}\cdot  \mathcal{T}^*
 \ll & \,\, |j|^{1/4}|d|^{-1/4}x^{\gamma/4+5/8} + |j|^{-1/4}|d|^{1/4}x^{1-\gamma/4}
                        \nonumber \\
     & \,\, + x^{89/100} +|j|^{1/6}|d|^{-1/6}x^{\gamma/6+3/4}.
\end{align*}
Combining the above three cases, we derive that
\begin{align*}
	x^{-\varepsilon}\cdot \mathcal{T}^*
	\ll & \,\, |j|^{1/6}|d|^{-1/6}x^{\gamma/6+3/4}+|j|^{-1/3}|d|^{1/3}x^{1-\gamma/3}\\
	&\,\, +|j|^{1/4}|d|^{-1/4}x^{\gamma/4+5/8} + |j|^{-1/4}|d|^{1/4}x^{1-\gamma/4} + x^{89/100},
\end{align*}
which completes the proof of this lemma.
\end{proof}

\section{Proof of Theorem~\ref{thm:main}}

\subsection{Initial construction and the main term}

Recall that $\gamma \defeq c^{-1}$ and $a \defeq \alpha^{-1}$. The set, which we sieve, is
$$
\cA \defeq \big\{n \le x^\gamma: p = \fl{n^c}, p = \fl{\alpha m + \beta}\big \},
$$
where $p$ is prime and $n,m$ are integers. For any $d \le D$, where $D$ is a fixed power of $x$, we estimate
$$
\cA_d \defeq \{m \in \cA: d\,|\, m \}.
$$
We know that $rd \in \cA$ if and only if
$$
p \le (rd)^c < p+1 \qquad \textrm{and}\qquad rd \leqslant x^\gamma
$$
and
$$
p = \fl{\alpha m + \beta}.
$$
The cardinality of $\cA_d$ is equal to the number of primes $p \le x$ for which the interval
$\big[p^\gamma d^{-1},(p+1)^\gamma d^{-1}\big)$ contains a natural number, where $p$ is in a Beatty sequence. Hence
\begin{align*}
|\cA_d| & = \Big|\Big\{ p \le x: -(p+1)^\gamma d^{-1} < -r \le -p^\gamma d^{-1}~\text{for some } r \in \NN,
             p = \fl{\alpha m + \beta}\Big\}\Big| \\
        & = \sum_{p \le x} \( \fl{-p^\gamma d^{-1}} - \fl{-(p+1)^\gamma d^{-1}} \)\chi_{\alpha, \beta}(p).
\end{align*}
By noting that
\begin{align*}
   & \,\, \fl{-p^\gamma d^{-1}} - \fl{-(p+1)^\gamma d^{-1}}
              \nonumber \\
 = & \,\, \gamma p^{\gamma - 1} d^{-1} + O(p^{\gamma-2}d^{-1}) + \psi\(-(p+1)^\gamma d^{-1} \) - \psi(-p^\gamma d^{-1}),
\end{align*}
and by Lemma~\ref{lem:Beatty}
$$
\chi_{\alpha, \beta} (p) = a + \psi(-a(p+1-\beta)) - \psi(-a(p-\beta)),
$$
then we obtain that
$$
|\cA_d| = \cX d^{-1} + \cS_1 + \cS_2 + \cS_3,
$$
where
\begin{align*}
  \cX &\defeq  \sum_{p \le x}a\(\gamma p^{\gamma-1}+O(p^{\gamma-2})\)=\frac{x^\gamma}{\alpha\log x}
     +O\(\frac{x^\gamma}{\log^2 x} \),
           \nonumber \\
\cS_1 &\defeq \sum_{p \le x} \big(  \gamma p^{\gamma-1} d^{-1} + O(p^{\gamma-2} d^{-1}) \big) \\
&\qquad\quad \times \big( \psi(-a(p+1-\beta))-\psi(-a(p-\beta)) \big), \\
\cS_2 &\defeq \sum_{p \le x} \big( \psi(-(p+1)^\gamma d^{-1})-\psi(-p^\gamma d^{-1}) \big) \\
&\qquad\quad \times \big( \psi(-a(p+1-\beta))-\psi(-a(p-\beta)) \big),	\\
\cS_3 &\defeq \sum_{p \le x} a\big( \psi\(-(p+1)^\gamma d^{-1} \) - \psi(-p^\gamma d^{-1}) \big). \\
\end{align*}
By Lemma~\ref{lem:Greaves}, we only need to show that, for any sufficiently small $\eps > 0$, there holds
\begin{equation}
\label{eq:main}
\sum_{d \le D} \big| |\cA_d| - \cX d^{-1} \big| \ll \cX x^{-\eps/3} \ll x^{\gamma-\eps/3}
\end{equation}
for sufficiently large $x$, which is sufficient to show that
\begin{equation}\label{dis-123}
\sum_{d\le D}\Big(\big|\cS_1\big|+\big|\cS_2\big|+\big|\cS_3\big|\Big)\ll x^{\gamma-\eps/3}.
\end{equation}
Our aim is to establish (\ref{dis-123}) with $D$ as large as possible, and in the following subsection we shall show that (\ref{dis-123}) holds when
\begin{equation*}
  D\ll x^{\gamma-11/12-\varepsilon}
\end{equation*}
for $\varepsilon>0$ sufficiently small. Suppose this has been done, and observe that for $R\geqslant13$
\begin{equation*}
   \gamma-\frac{11}{12}>\frac{8}{8R-1} \qquad \textrm{if and only if}\qquad \gamma>\frac{88R+85}{96R-12}.
\end{equation*}
Then, for any fixed $\gamma\in(\frac{88R+85}{96R-12},1)$ and $\vartheta\in(\frac{8}{8R-1},\gamma-\frac{11}{12})$, the left--hand
side of (\ref{dis-123}) with $D:=x^\vartheta$ implies the bound
\begin{equation*}
   \sum_{d \le D} \big| |\cA_d| - \cX d^{-1} \big|  \ll \frac{\cX}{\log^2\cX},
\end{equation*}
and thus we can apply the weighted sieve in the form of Lemma \ref{lem:Greaves} with the choices of the parameters in \cite[pp.~174--175]{Greaves} as follows
$$
R \geqslant 13,\qquad \delta_R \defeq 0.124820, \qquad g\defeq \frac{8R-1}{8}.
$$
For sufficiently large $x$, by noting that $g<R-\delta_R$ and the inequality $a<D^g$ holds for all
$a\in\cA$ since $\vartheta g>1$. Thus, the conditions of Lemma \ref{lem:Greaves} are met, and
we conclude that $\cA$ contains at least $\gg\cX/\log\cX $ numbers with at most $R$ prime factors. This yields the statement of the theorem for all $c$ in the interval $(1,\frac{96R-12}{88R+85})$.

Next, we estimate the contribution of the upper bound generated by $\cS_1,\cS_2$ and $\cS_3$ separately.

\subsection{Estimation of $\cS_1$}
We write $\cS_1 = \cS_{11} + O(\cS_{12})$, where
$$
\cS_{11} \defeq \sum_{p \le x} \gamma p^{\gamma-1} d^{-1}  \big( \psi(-a(p+1-\beta))-\psi(-a(p-\beta)) \big)
$$
and
$$
\cS_{12} \defeq \sum_{p \le x} p^{\gamma-2} d^{-1}.
$$
We choose that
$$
H \defeq x^\eps.
$$
By the Vaaler's approximation, i.e. Lemma~\ref{lem:Vaaler}, we have
\begin{align}\label{dis-012}
     & \,\, \psi(-a(p+1-\beta))-\psi(-a(p-\beta))
                \nonumber \\
  = & \,\, \sum_{0<|h|\leqslant H}a_h \big(\e (ah(p+1-\beta))-\e(ah(p-\beta))\big)
                \nonumber \\
    & \,\, +O\Bigg(\sum_{|h|\leqslant H}b_h\big(\e(ah(p+1-\beta))+\e(ah(p-\beta))\big)\Bigg),
\end{align}
and thus $\cS_{11} = \cS_{13} + O(\cS_{14})$, where
$$
\cS_{13}\defeq \sum_{p \le x}\sum_{0<|h|\le H}\gamma p^{\gamma-1}d^{-1}a_h\big( \e(ah(p+1-\beta))-\e(ah(p-\beta))\big)
$$
and
$$
\cS_{14}\defeq\sum_{p\leqslant x}\sum_{|h|\le H}p^{\gamma-1}d^{-1} b_h\big(\e(ah(p+1-\beta))+\e(ah(p-\beta))\big).
$$
The contribution of $\cS_{14}$ from $h=0$ is
$$
\ll \sum_{p \le x} p^{\gamma-1} d^{-1} H^{-1} \ll H^{-1} d^{-1} x^\gamma \log^{-1} x.
$$
Then we have
\begin{align*}
\sum_{d \le D} \big| H^{-1} d^{-1} x^\gamma \log^{-1} x \big| \ll x^{\gamma - \eps/3},
\end{align*}
for any $D \ll x^A$, where $A$ is an arbitrary positive real number. Now we consider $\cS_{13}$, since the upper bound estimate of $\cS_{14}$ from $h \neq 0$ is similar to that of $\cS_{13}$. By Lemma \ref{lem:index} and a splitting argument, it suffices to prove that
\begin{align}\label{eq:S13}
	& \frac{1}{\log x}\sum_{d \le D}d^{-1}\max_{x/2<t\leqslant x} \Bigg| \sum_{x/2<n\leqslant t}\Lambda(n)n^{\gamma-1}
      \sum_{0<|h|\leqslant H}a_{h}   \nonumber \\
	&\qquad\qquad\times\big(\e(ah(n+1-\beta))-\e(ah(n-\beta))\big)\Bigg|\ll x^{\gamma-\eps/3}.
\end{align}
Define
\begin{equation}\label{eq:theta}
  \theta_{h}\defeq\e(ah)-1.
\end{equation}
Then, by partial summation and the trivial estimate $\theta_{h} \ll 1$, the left--hand side of \eqref{eq:S13} is
\begin{align*}
  \ll & \,\, \max_{x/2<t\leqslant x}\Bigg|\sum_{x/2<n\leqslant t}\Lambda(n)n^{\gamma-1}\sum_{0<|h|\leqslant H}a_{h}
             \theta_{h}\e(ah(n-\beta))\Bigg|
                   \nonumber \\
  \ll & \,\, \max_{x/2<t\leqslant x}\Bigg|\int_{\frac{x}{2}}^tu^{\gamma-1}\mathrm{d}\Bigg(\sum_{x/2<n\leqslant u}
             \Lambda(n)\sum_{0<|h|\leqslant H}a_{h}\theta_{h}\e(ah(n-\beta))\Bigg)\Bigg|
                   \nonumber \\
  \ll & \,\, x^{\gamma-1}\times\max_{x/2<u\leqslant x}\Bigg|\sum_{x/2<n\leqslant u}\Lambda(n)\sum_{0<|h|\leqslant H}
             a_{h}\theta_{h}\e(ah(n-\beta))\Bigg|
                   \nonumber \\
  \ll & \,\, x^{\gamma-1}\times\max_{x/2<u\leqslant x}\sum_{0<|h|\leqslant H}\frac{1}{|h|}
             \Bigg|\sum_{x/2<n\leqslant u}\Lambda(n)\e(ahn)\Bigg|.
\end{align*}
Moreover, it follows from Lemma \ref{lem:ttr} that the left--hand side of \eqref{eq:S13} holds
\begin{align*}
\ll & \,\, x^{\gamma-1}\sum_{0<h\le H}h^{-1}(h^{1/2}x^{1-1/(2\tau)+\eps}+x^{1-\eps}) \\
\ll & \,\, x^{\gamma-1}(H^{1/2}x^{1-1/(2\tau)+\eps}+x^{1-\eps}\log H) \\
\ll & \,\, x^{\gamma-1/(2\tau)+3\eps/2}+x^{\gamma-\eps}\ll x^{\gamma-\eps/3},
\end{align*}
uniformly for any $D \ll x^A$, where $A$ is an arbitrary positive real number. For the contribution of $\cS_{12}$ to (\ref{dis-123}), we have
\begin{equation*}
 \ll \sum_{d\leqslant D}\frac{1}{d}\sum_{p\leqslant x}p^{\gamma-2}\ll x^{\gamma-1+\varepsilon}\ll x^{\gamma-\varepsilon/3}.
\end{equation*}

\subsection{Estimation of $\cS_3$}

The estimation of $S_3$ is the same as the proof of~\cite[p. 360, equation (5)]{BBGY}. Based on the same method, we have
$$
\sum_{d\le D}\big|\cS_3\big|\ll x^{\gamma-\eps/3},
$$
with
$$
D \ll x^{\gamma - \frac{380}{441}-\eps},
$$
from~\cite[p. 366, equation (23)]{BBGY}.

\subsection{Estimation of $\cS_2$}
 By the Vaaler's approximation, i.e. Lemma \ref{lem:Vaaler}, and taking $J \defeq x^{1-\gamma + \eps} d$,    we have that
\begin{align}\label{eq:gamma}
&\qquad \psi(-(p+1)^\gamma d^{-1})-\psi(-p^\gamma d^{-1}) \nonumber \\
&= \sum_{0 < |j| \le J} a_{j} \big( \e(j (p+1)^\gamma d^{-1}) - \e(j p^\gamma d^{-1}) \big) \nonumber \\
&\qquad + O\( \sum_{|j| \le J} b_{j} \big( \e(j (p+1)^\gamma d^{-1}) + \e(j p^\gamma d^{-1}) \big) \),
\end{align}
with the coefficients satisfying
$$
a_{j} \ll |j|^{-1} \mand b_{j} \ll J^{-1}.
$$
By (\ref{dis-012}) and (\ref{eq:gamma}), it is easy to see that
\begin{equation*}
\cS_2=\cS_{21}+O(\cS_{22}+\cS_{23}+\cS_{24}),
\end{equation*}
where
\begin{align*}
\cS_{21} &\defeq \sum_{p \le x} \sum_{0 < |j| \le J} a_{j} \big( \e(j (p+1)^\gamma d^{-1}) - \e(j p^\gamma d^{-1}) \big) \\
&\qquad \times \sum_{0 < |h| \le H} a_h \( \e(ah(p+1-\beta)) - \e(ah(p-\beta)) \), \\
\cS_{22} &\defeq \sum_{p \le x} \sum_{0 < |j| \le J} a_{j} \big( \e(j (p+1)^\gamma d^{-1}) - \e(j p^\gamma d^{-1}) \big) \\
&\qquad \times \sum_{|h| \le H} b_h \( \e(ah(p+1-\beta)) + \e(ah(p-\beta)) \), \\
\cS_{23} &\defeq \sum_{p \le x} \sum_{|j| \le J} b_{j} \big( \e(j (p+1)^\gamma d^{-1}) + \e(j p^\gamma d^{-1}) \big) \\
&\qquad \times \sum_{0 < |h| \le H} a_h \( \e(ah(p+1-\beta)) - \e(ah(p-\beta)) \), \\
\cS_{24} &\defeq \sum_{p \le x} \sum_{|j| \le J} b_{j} \big( \e(j (p+1)^\gamma d^{-1}) + \e(j p^\gamma d^{-1}) \big) \\
&\qquad \times \sum_{|h| \le H} b_h \( \e(ah(p+1-\beta)) + \e(ah(p-\beta)) \).
\end{align*}
\subsubsection{Estimation of $\cS_{21}$}
By Lemma~\ref{lem:index} and a splitting argument, it is sufficient to prove that
\begin{align}\label{eq:S21}
  & \,\, \sum_{d \le D}\max_{x/2<t\leqslant x}\Bigg|\sum_{x/2<n\leqslant t}\Lambda(n)\sum_{0<|j|\leqslant J}a_{j}
         \big(\e(j(n+1)^\gamma d^{-1})-\e(jn^\gamma d^{-1})\big)
               \nonumber\\
  & \qquad \,\, \times \sum_{0<|h|\leqslant H}a_h\big(\e(ah(n+1-\beta))-\e(ah(n-\beta))\big)\Bigg|\ll x^{\gamma-\eps/3},
\end{align}
Define
\begin{equation*}
  \phi_j(t):=\e\big(jd^{-1}((t+1)^\gamma-t^\gamma)\big)-1.
\end{equation*}
Then we have
\begin{equation*}
  \phi_j(t)\ll |j||d|^{-1}t^{\gamma-1}\qquad \textrm{and}\qquad
  \frac{\partial\phi_j(t)}{\partial t}\ll |j||d|^{-1}t^{\gamma-2}.
\end{equation*}
It follows from the above estimate, (\ref{eq:theta}), Lemma \ref{lem:exp} and partial summation that
the left--hand side of (\ref{eq:S21}) is
\begin{align*}
\ll & \,\, \sum_{d \le D}\max_{x/2<t\leqslant x}\Bigg|\sum_{0<|j|\leqslant J}a_j\sum_{x/2<n\leqslant t}\Lambda(n)
           \phi_j(n)\e(jd^{-1}n^\gamma)
                  \nonumber \\
    & \,\, \qquad \qquad\times\sum_{0<|h|\leqslant H}a_h\big(\e(ah(n+1-\beta))-\e(ah(n-\beta))\big)\Bigg|
                  \nonumber \\
\ll & \,\, \sum_{d \le D}\max_{x/2<t\leqslant x}\sum_{0<|j|\leqslant J}\frac{1}{|j|}\Bigg|\sum_{x/2<n\leqslant t}
           \Lambda(n)\phi_j(n)\e(jd^{-1}n^\gamma)
                  \nonumber \\
    & \,\, \qquad \qquad\times\sum_{0<|h|\leqslant H}a_h\big(\e(ah(n+1-\beta))-\e(ah(n-\beta))\big)\Bigg|
                  \nonumber \\
\ll & \,\, \sum_{d \le D}\max_{x/2<t\leqslant x}\sum_{0<|j|\leqslant J}\frac{1}{|j|}\Bigg|\int_{\frac{x}{2}}^t\phi_j(\nu)
           \mathrm{d}\Bigg(\sum_{x/2<n\leqslant\nu}\Lambda(n)\e(jd^{-1}n^\gamma)
                  \nonumber \\
    & \,\, \qquad \qquad\times\sum_{0<|h|\leqslant H}a_h\big(\e(ah(n+1-\beta))-\e(ah(n-\beta))\big)\Bigg)\Bigg|
                  \nonumber \\
\ll & \,\, \sum_{d \le D}\max_{x/2<t\leqslant x}\sum_{0<|j|\leqslant J}\frac{1}{|j|}\Big|\phi_j(t)\Big|\Bigg|
           \sum_{x/2<n\leqslant t}\Lambda(n)\e(jd^{-1}n^\gamma)
                  \nonumber \\
    & \,\, \qquad \qquad\times\sum_{0<|h|\leqslant H}a_h\big(\e(ah(n+1-\beta))-\e(ah(n-\beta))\big)\Bigg|
                  \nonumber \\
    & \,\, +\sum_{d \le D}\max_{x/2<t\leqslant x}\int_{\frac{x}{2}}^t\sum_{0<|j|\leqslant J}\frac{1}{|j|}
           \Bigg|\frac{\partial\phi_j(\nu)}{\partial\nu}\Bigg|\Bigg|\sum_{x/2<n\leqslant\nu}\Lambda(n)\e(jd^{-1}n^\gamma)
                  \nonumber \\
    & \,\, \qquad \qquad\times\sum_{0<|h|\leqslant H}a_h\big(\e(ah(n+1-\beta))-\e(ah(n-\beta))\big)\Bigg|\mathrm{d}\nu
                  \nonumber \\
\ll & \,\, x^{\gamma-1}\times\sum_{d \le D}d^{-1}\sum_{0<|j|\leqslant J}\max_{x/2<t\leqslant x}
           \Bigg|\sum_{x/2<n\leqslant t}\Lambda(n)\e(jd^{-1}n^\gamma)
                  \nonumber \\
    & \,\, \qquad \qquad\times\sum_{0<|h|\leqslant H}a_h\big(\e(ah(n+1-\beta))-\e(ah(n-\beta))\big)\Bigg|
                  \nonumber \\
\ll & \,\, x^{\gamma-1}\times\sum_{d \le D}d^{-1}\sum_{0<|j|\leqslant J}\max_{x/2<t\leqslant x}
           \Bigg|\sum_{0<|h|\leqslant H}a_h\theta_h\sum_{x/2<n\leqslant t}\Lambda(n)\e(jd^{-1}n^\gamma+ahn)\Bigg|
                  \nonumber \\
\ll & \,\, x^{\gamma-1}\times\sum_{0<|h|\leqslant H}\frac{1}{|h|}\sum_{d \le D}d^{-1}\sum_{0<|j|\leqslant J}
           \max_{x/2<t\leqslant x}\Bigg|
           \sum_{x/2<n\leqslant t}\Lambda(n)\e(jd^{-1}n^\gamma+ahn)\Bigg|
                  \nonumber \\
\ll & \,\, x^{\gamma-1+\varepsilon}\times\sum_{0<|h|\leqslant H}\frac{1}{|h|}\sum_{d \le D}d^{-1}\sum_{0<|j|\leqslant J}
           \Big(|j|^{1/6}|d|^{-1/6}x^{\gamma/6+3/4}+|j|^{-1/3}|d|^{1/3}x^{1-\gamma/3}
                  \nonumber \\
    & \,\, \qquad \qquad+|j|^{1/4}|d|^{-1/4}x^{\gamma/4+5/8}+|j|^{-1/4}|d|^{1/4}x^{1-\gamma/4}+x^{89/100}\Big)
                  \nonumber \\
\ll & \,\, x^{\gamma-1+\varepsilon}\Big(Dx^{23/12-\gamma}+Dx^{5/3-\gamma}+Dx^{15/8-\gamma}
           +Dx^{7/4-\gamma}+Dx^{189/100-\gamma}\Big)
                  \nonumber \\
\ll & \,\, Dx^{11/12+\varepsilon}\ll x^{\gamma-\varepsilon},
\end{align*}
provided that $D\ll x^{\gamma-11/12-\varepsilon}$, which is worse than the bound of $D \ll x^{\gamma - \frac{380}{441}-\eps}$ from the estimation of $\cS_{3}$.

\subsubsection{Estimation of $\cS_{22}$}

The contribution from $h = 0$ is
$$
\ll H^{-1} \sum_{p \le x} \sum_{0 < |j| \le J} a_{j} \big( \e(j (p+1)^\gamma d^{-1}) - \e(j p^\gamma d^{-1}) \big),
$$
which can be estimated by the same proof of \cite[p. 362, equation (9)]{BBGY} with
$$
D \ll x^{\gamma - \frac{380}{441}-\eps}.
$$
The contribution from $h \neq 0$ can be bounded by same method of $S_{21}$.

\subsubsection{Estimation of $\cS_{23}$}

The contribution from $j = 0$ is
$$
\ll \sum_{p \le x} J^{-1} \sum_{0 < |h| \le H} a_h \( \e(ah(p+1-\beta)) - \e(ah(p-\beta)) \),
$$
which can be estimated by the same method of $\cS_{1}$ with $D \ll x^A$. The contribution from $j \neq 0$ can be bounded by same method of $S_{21}$.

\subsubsection{Estimation of $\cS_{24}$}

The contribution from $h = j = 0$ is
$$
\ll \sum_{p \le x} H^{-1} J^{-1} \ll x^{\gamma - \eps} d^{-1}.
$$
Then we have
$$
\sum_{d \le D} x^{\gamma - \eps} d^{-1} \ll x^{\gamma - \eps/3}.
$$
for any $D \ll x^A$, where $A$ is an arbitrary positive real number. The contribution from $h = 0$ and $j \neq 0$ is
$$
\ll \sum_{p \le x} H^{-1} \sum_{0 < |j| \le J} b_{j} \big( \e(j (p+1)^\gamma d^{-1}) + \e(j p^\gamma d^{-1}) \big),
$$
which can be bounded by the same method of $\cS_{22}$. The contribution from $h \neq 0$ and $j = 0$ is
$$
\ll \sum_{p \le x} J^{-1} \sum_{|h| \le H} b_h \big( \e(ah(p+1-\beta)) + \e(ah(p-\beta)) \big),
$$
which can be treated by the same method of $\cS_{1}$. The contribution from $h \neq 0$ and $j \neq 0$ is
\begin{align*}
&\sum_{p \le x} \sum_{0 < |j| \le J} b_{j} \big( \e(j (p+1)^\gamma d^{-1}) + \e(j p^\gamma d^{-1}) \big) \\
&\qquad \times \sum_{0 < |h| \le H} b_h \( \e(ah(p+1-\beta)) + \e(ah(p-\beta)) \),
\end{align*}
which can be estimated by the same method of $\cS_{21}$.

\section{Acknowledgement}
The first author is supported by the National Natural Science Foundation of China (Grant No. 11901447), the
China Postdoctoral Science Foundation (Grant No. 2019M653576), the Natural Science Foundation of Shaanxi
Province (Grant No. 2020JQ009), and the China Scholarship Council (Grant No. 201906285058).

The second author is supported by the National Natural Science Foundation of China
(Grant No. 11901566, 11971476, 12071238), the Fundamental Research Funds for the Central Universities
(Grant No. 2021YQLX02), and National Training Program of Innovation and Entrepreneurship for Undergraduates
(Grant No. 202107010).

The third author is supported by the National Natural Science Foundation of China (Grant No. 12001047, 11971476),
and the Scientific Research Funds of Beijing Information Science and Technology University (Grant No. 2025035).

\end{document}